\documentclass[11pt]{amsart}
\usepackage{geometry}
\usepackage{amsmath, amssymb}
\usepackage{amsmath, amscd}
\usepackage{amsfonts}
\usepackage{amsrefs}

\usepackage{hyperref}
\usepackage{cleveref}
\usepackage{mathrsfs}
\usepackage{setspace}
\usepackage[dvipsnames]{xcolor}
\usepackage{pgffor}

\usepackage{pb-diagram} 
\usepackage{pb-xy}
\usepackage{graphicx}

\usepackage{amsthm}
\usepackage{subcaption}

\usepackage{color}
\usepackage{float}
\usepackage{xcolor}

\usepackage{xfrac}

\usepackage{tikz}\usetikzlibrary{knots}

\allowdisplaybreaks

\makeatletter
\newtheorem*{rep@theorem}{\rep@title}
\newcommand{\newreptheorem}[2]{%
\newenvironment{rep#1}[1]{%
 \def\rep@title{#2 \ref{##1}}%
 \begin{rep@theorem}}%
 {\end{rep@theorem}}}
\makeatother

\newtheorem{theorem}{Theorem}
\newreptheorem{theorem}{Theorem}

\newtheorem{proposition}[theorem]{Proposition}
\newtheorem{lemma}[theorem]{Lemma}

\newtheorem*{theorem*}{Theorem}
\newtheorem*{proposition*}{Proposition}
\theoremstyle{remark}

\newtheorem{definition}[theorem]{Definition}

\newcommand{\lk}{\operatorname{lk}}

\newcommand{\M}{\mathcal{M}}
\newcommand{\N}{\mathbb{N}}
\newcommand{\Z}{\mathbb{Z}}
\newcommand{\R}{\mathbb{R}}

\newcommand{\bdry}{\partial}

\newcommand{\Sum}{\displaystyle\sum}

\newcommand{\Oint}{\displaystyle\oint}

\newcommand{\ceil}[1]{\left\lceil #1 \right\rceil}
\newcommand{\floor}[1]{\left\lfloor #1 \right\rfloor}

\newcommand{\pref}[1]{(\ref{#1})}

 \newcommand{\comment}[1]{}
\begin{document}

\title[Bounds on the clasp number.]{$C$-complexes, Clasp Number, and Triple Linking Number}

\author{Christopher William Davis}
\address{Department of Mathematics, University of Wisconsin-Eau Claire, Hibbard Humanities Hall 508,  Eau Claire WI 54702-4004}
\email{daviscw@uwec.edu}

\author{David Lawrence}
\address{Department of Mathematics, University of Wisconsin-Eau Claire, Hibbard Humanities Hall 508,  Eau Claire WI 54702-4004}
\email{lawrencd8578@uwec.edu}

\author{Jack Paulsen}
\address{Department of Mathematics, University of Wisconsin-Eau Claire, Hibbard Humanities Hall 508,  Eau Claire WI 54702-4004}
\email{paulsejd4608@uwec.edu}

\author{Nathan Phillips}
\address{Department of Mathematics, University of Wisconsin-Eau Claire, Hibbard Humanities Hall 508,  Eau Claire WI 54702-4004}
\email{phillins0745@uwec.edu}

\date{\today}

\subjclass[2010]{}

\keywords{}

%

\begin{abstract} A C-complex is a union of Seifert surfaces for the components of a link which intersect each other in clasps. The clasp number of a link is the minimal number of clasps amongst all C-complexes it bounds. It gives a measure of how far a link is form being a boundary link. This paper provides a new lower bound for the number of clasps of all C-complexes bounded by a given 3-component link improving results of Amundsen-Anderson-D.-Guyer. Furthermore, we construct links that achieve these bounds. 
In order to do so, we express the triple linking numbers as the area bounded by three curves in the plane, called word curves, and then perform the geometry and discrete optimization needed to minimize the length of these curves.
\end{abstract}


\maketitle

\section{Introduction}

In his thesis work Cooper \cite{CooperThesis, Cooper82} defined an object called a C-complex (or clasp-complex), see for example Figure~\ref{fig: Sample computation}.  These objects give useful generalizations of Seifert surfaces to the setting of links.  Informally a C-complex for a link $L$ consists of a union of of Seifert surfaces for the components of $L$, where these individual surfaces are allowed to intersect each other, but only in clasps.  We recall the precise definition in Section~\ref{background}.  He used these objects to extend the classical computation of the Alexander module to the setting of links.  Further work championed by Cimasoni \cite{Cimasoni2004} starting in the early 2000's reinvigorated the study of C-complexes. Since then they have been used as an avenue to understand link theory.  For example, they can be used to understand algebraic invariants like the signature, Alexander polynomial and Blanchfield form \cite{CimFlo2008, CimTur07,  CimConZac, CFT18, DFL24} and they motivate a notion of S-equivalence \cite{CimTur07, DMOP21}.

\begin{figure}
         \centering
         \begin{tikzpicture}
        \node at (0,0){\includegraphics[width=.25\textwidth]{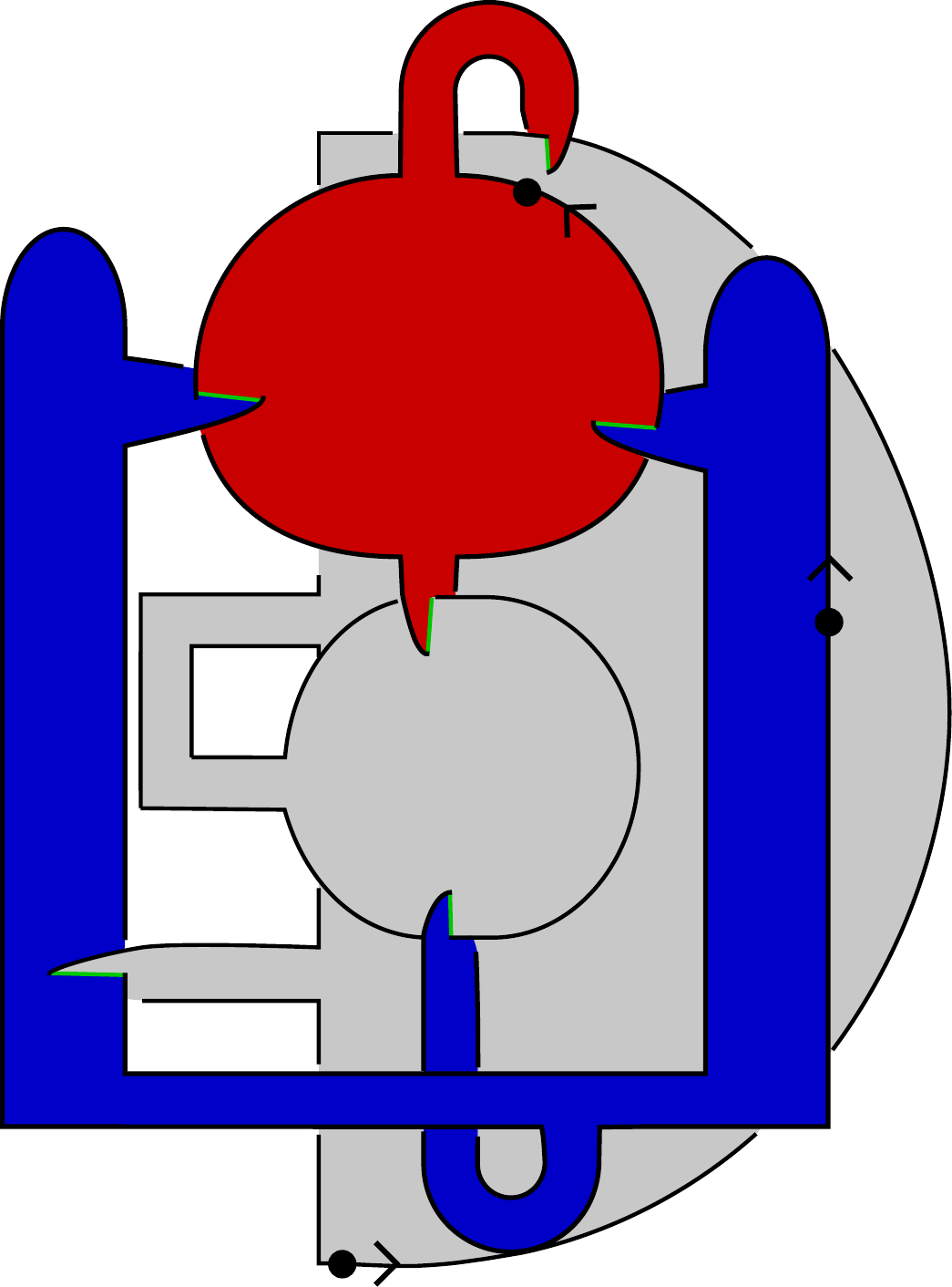}};
        \node at (0,3){$L_1$};
        \node at (-1,-3){$L_3$};
        \node at (-2,2){$L_2$};
         \end{tikzpicture}

        \caption{A link with $\mu_{123}(L)=3$ bounding a C-complex with 6 clasps.  Our work will prove that every $C$-complex bounded by any link with $\mu_{123}(L)=3$ has at least 6 clasps.}
        \label{fig: Sample computation}
\end{figure}
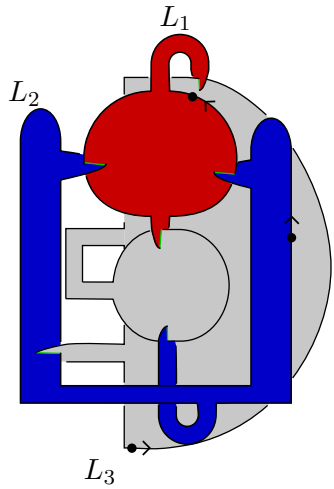

There are several natural measures of the complexity of a link coming from a C-complex.  One can ask about the sum of the genera of its components, the first Betti number of its union, or the number of clasps.  We focus on the latter. The \textbf{clasp number} of $L$ is the minimal number of clasps amongst all C-complexes bounded by $L$.  We denote it by $C(L)$.  In \cite[Theorem 2]{AADG20} the first author, along with Amundsen, Anderson and Guyer, reveals that for 2-component links the clasp number is basically determined by the linking number.  In that paper also appears a lower bound on $C(L)$ for 3-component link.    In the current work, we build on this idea and find a stronger bound on $C(L)$ for 3-component links.  Our bound is sharp.  

 Our work relies on Milnor's triple linking number and so we begin by recalling some of its history.  In the 1950's \cite{Milnor1950} Milnor built a new infinite family of invariants which generalize linking number.  While the original definition relies on an in-depth algebraic analysis of lower central series quotients and the Magnus embedding, they have been reformulated in a myriad of ways.  In \cite{CochranBook}, Cochran explained how to compute Milnor's invariants from iterated intersections of embedded surfaces.

 In \cite{MellorMelvin2003} Mellor-Melvin built on Cochran's work to translate this computation to a combinatorial formulation for Milnor's triple linking number, $\mu_{123}$, in terms of a C-complex.  In truth, \cite{MellorMelvin2003} allows for any system of Seifert surfaces, but we will restrict it to C-complexes.  This formulation involves only the clasps of that C-complex and so it is not surprising that it says something about the clasp number.  According to \cite[Theorem 5]{AADG20} $C(L)\ge 2\ceil{2\sqrt{|\mu_{123}(L)/3}|}$, where $\ceil{-}$ indicates the ceiling of a real number.   The current work builds on this, producing a better lower bound on the clasp number.

\begin{theorem}\label{thm: calc} Let \(L=L_1\cup L_2\cup L_3\) be a 3-component link with vanishing pairwise linking numbers. Then, \(C(L)\geq2\ceil{\sqrt{3|\mu_{123}(L)|}}\). 
\end{theorem}

As an immediate application, the link of Figure~\ref{fig: Sample computation} has  C-complex with 6 clasps, so $C(L)\le 6$.   It also has $\mu_{123}(L)=3$, so $C(L)\ge 6$.  In Section~\ref{background}, we recall the technique developed in~\cite{AADG20} to compute the triple linking number in terms of an area computation in the plane.  In Section~\ref{sect: lower bound} we prove Theorem~\ref{thm: calc} by analyzing the relationship between the area and perimeter of these regions.

  In contrast to \cite[Theorem 5]{AADG20} our bound is sharp.  

\begin{theorem}\label{thm:realization}
For any $M\in \Z$,  there exists a link L such that $\mu_{123}(L)=M$ and \(C(L)=2\ceil{\sqrt{3|M|}}.\)\end{theorem} 

The links we produce are explicit.  See Figure~\ref{fig: Optimal example}.  In Section~\ref{sect:examples} we motivate these examples and prove Theorem~\ref{thm:realization}.

\begin{figure}
     \centering
     \begin{subfigure}[t]{0.45\textwidth}
         \centering
         \begin{tikzpicture}
         \node at (0,.4){\includegraphics[width=\textwidth]{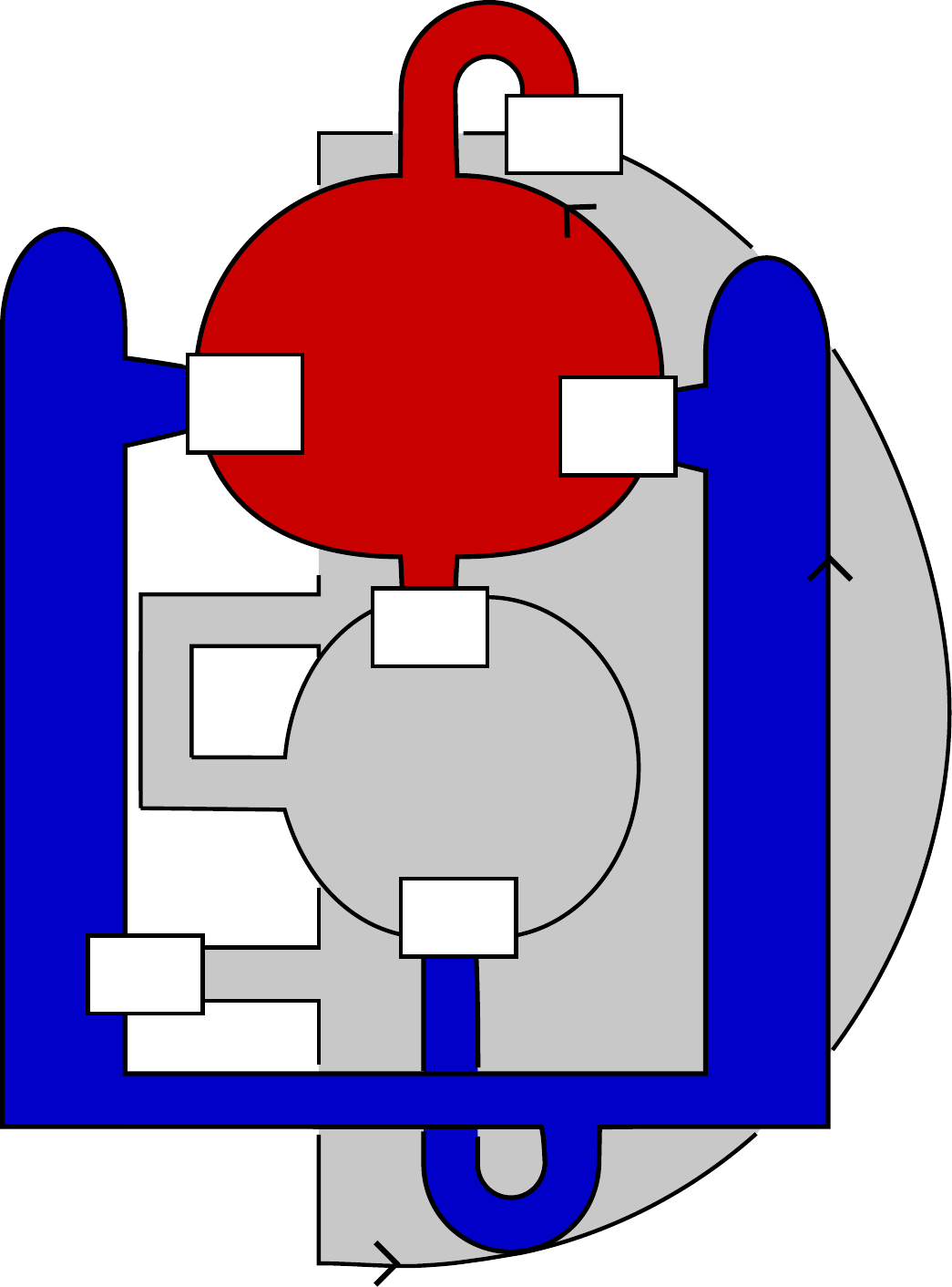}};
         \node at (.6,4.3) {$\beta$};
         \node at (1.1,2.1) {$\alpha$};
         \node at (-.35,.5) {$-\beta$};
         \node at (-1.85,2.2) {$-\alpha$};
         \node at (-.3,-1.8) {$\gamma$};
         \node at (-2.5,-2.2) {$-\gamma$};

        \node at (1,5){$L_1$};
        \node at (-1.5,-4){$L_3$};
        \node at (2.55,3.555){$L_2$};
         
         \end{tikzpicture}
         \caption{If $M=\alpha\beta+\beta\gamma+\gamma\alpha$ then the link above has $\mu_{123}(L)=M$ and $C(L) = 2\ceil{\sqrt{3M}}$.}
         \label{fig: SimplestCase}
     \end{subfigure}
     \hfill
     \begin{subfigure}[t]{0.5\textwidth}
         \begin{tikzpicture}
         \node at (0,0){\includegraphics[width=\textwidth]{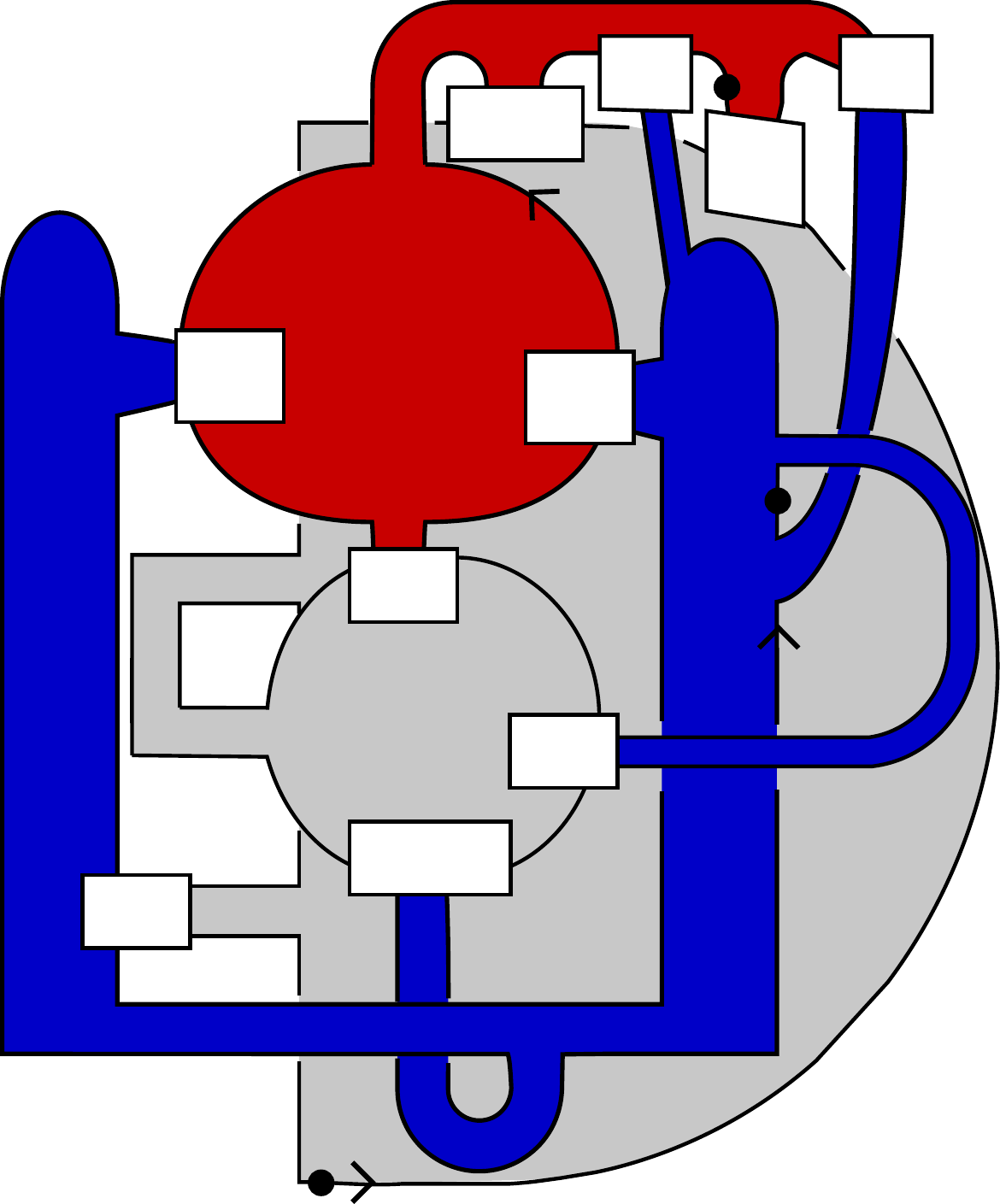}};
         \node at (.1,4) {$\beta-\ell_1$};
         \node at (1.2, 4.3) {$+1$};
         \node at (2.1, 3.6) {$\ell_1$};
         \node at (3.2, 4.3) {$-1$};
         \node at (.7,1.7) {$\alpha$};
         \node at (-.8,.1) {$-\beta$};
         \node at (-2.2,1.9) {$-\alpha$};
         \node at (-3,-2.6) {$-\gamma$};
         \node at (-.6,-2.2) {$\gamma-\ell_2$};
         \node at (.5,-1.2) {$\ell_2$};
         
        \node at (-1.3,4.5){$L_1$};
        \node at (-2,-4.3){$L_3$};
        \node at (-3.5,3.5){$L_2$};
         \end{tikzpicture}
         \caption{If $M=\alpha\beta+\beta\gamma+\gamma\alpha+\ell_1+\ell_2$ with $1\le \ell_1< \beta$ and $0\le \ell_2<\gamma$ then the link above has $\mu_{123}(L)=M$ and $C(L) = 2\ceil{\sqrt{3M}}$.}
         \label{fig: First Case}
     \end{subfigure}

        \caption{Links satisfying $C(L) = 2\ceil{\sqrt{3|\mu_{123}(L)|}}$. Here, $\mu_{123}(L)\in \N$, and  $(\alpha, \beta,\gamma)\in \{(k,k,k), (k,k,k+1), (k,k+1,k+1)\}$ with $k\in \N\cup\{0\}$.  A box containing a positive (resp. negative) integer should be interpreted as containing that many positive (resp. negative) clasps.  
        }
        \label{fig: Optimal example}
\end{figure}

\section{Background C-complexes and the triple linking number}\label{background}

An $n$-component link $L=L_1\cup\dots\cup L_n$ is an isotopy class of ordered oriented disjoint unions of smoothly embedded circles in the 3-sphere, $S^3$.  Throughout this paper all links will be 3-component links, except when we explicitly state otherwise.

\begin{definition}[See Section 2.1 of \cite{CimFlo2008}]\label{defn:C-cplx}
Let $L=L_1\cup L_2\cup L_3$ be a 3-component link.  A C-complex for $L$ is a union of three smoothly embedded oriented compact surfaces $F=F_1\cup F_2\cup F_3$ which intersect transversely and satisfy the following conditions:
\begin{enumerate}
\item For each $i$, $\bdry F_i=L_i$ and $F_i$ has no closed components.
\item For each $i\neq j$ $F_i\cap F_j$ consists of a (possibly empty) union of disjoint embedded arcs each with one endpoint in $L_i$ and one in $L_j$.  We call these clasps, see Figure~\ref{fig: clasps}.
\item $F_1\cap F_2\cap F_3=\emptyset$.
\end{enumerate}
\end{definition}

\begin{figure}
     \centering
     \begin{subfigure}[b]{0.25\textwidth}
         \centering
         \begin{tikzpicture}
        \node at (0,0){\includegraphics[width=.75\textwidth]{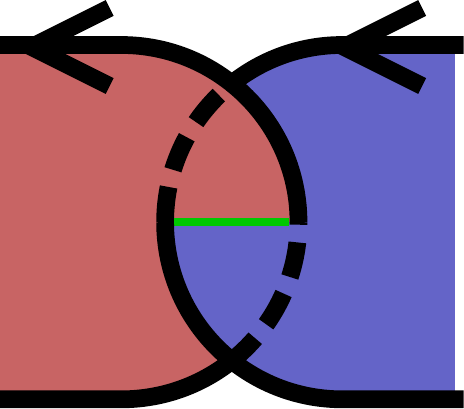}};
         \end{tikzpicture}
         \caption{A positive clasp}
         \label{fig: positive clasp}
     \end{subfigure}
     \begin{subfigure}[b]{0.25\textwidth}
         \centering
         \begin{tikzpicture}
        \node at (0,0){\includegraphics[width=.75\textwidth]{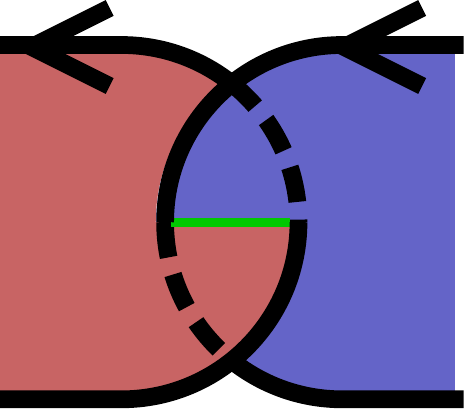}};
         \end{tikzpicture}
         \caption{A negative clasp}
         \label{fig: negative clasp}
     \end{subfigure}

        \caption{}
        \label{fig: clasps}
\end{figure}

The \textbf{clasp number} of $L$, denoted $C(L)$, is defined to be the minimal number of clasps amongst all C-complexes bounded by $L$.  

Let $i\neq j$ and $F$ be a C-complex for $L$. A point in the intersection of $L_i$ and $F_j$ is said to be positive if the positive tangent direction to $L_i$ agrees with the positive normal direction to $F_j$.  Otherwise this point of intersection is negative.   For any clasp $c\subseteq F_i\cap F_j$ in a C-complex $F$, we say that $c$ is a positive (or negative) clasp if the intersections points in $L_i\cap F_j$ (or $F_i\cap L_j$) is positive (or negative).  
The linking number, $\lk(L_i,L_j)$, is given by the signed count of clasps in $F_i\cap F_j$ and so $|\lk(L_1, L_2)|+|\lk(L_1, L_3)|+|\lk(L_2, L_3)|\le C(L)$.  Noteworthily, for any 2-component link, $C(L_1\cup L_2) = |\lk(L_1,L_2)|$, as long as $\lk(L_1,L_2)\neq 0$  \cite{AADG20}.  For 3-component links, there is another obstruction.  Mellor-Melvin \cite{MellorMelvin2003} shows how to compute the triple linking number, $\mu_{123}(L)$, in terms of the clasps of a C-complex.  We recall this formula below along with an illustrative example.  See also \cite[Section 3.1]{DaRo2017}.

  Let $L$ be a 3-component link and $F$ be a C-complex for $L$.  We begin by defining words called \textbf{clasp-words} which record the order that each component of $L$ encounters the component of $F$.  Fix a base-point $p_1$ on $L_1$, and define a word $w_1(F)$ in $x_2$, $x_2^{-1}$, $x_3$, and $x_3^{-1}$ as follows.  Start at $p_1$ and follow $L_1$ in the positive direction.  Each time $L_1$ intersects $F_2$ (or $F_3$) with sign $\epsilon = \pm1$, record $x_2^\epsilon$ (or $x_3^\epsilon$).  Pick basepoints $p_2$ and $p_3$ on $L_2$ and $L_3$ and define $w_2(F)$ and $w_3(F)$ analogously.  For the sake of example, given the C-complex $F$ and base-points in Figure~\ref{fig: Sample computation},  
\begin{equation}\label{eqn:claspwords}
w_1=x_3x_2^{-1}x_3^{-1}x_2,~
w_2=x_1x_3^{-1}x_1^{-1}x_3,~
w_3=x_1x_2x_1^{-1}x_2^{-1}.
\end{equation}

For any permutation $i,j,k$ of $1,2,3$ define $e_{ij}(w_k)$  by counting with sign how many instances of $x_i$ appear before an instance of $x_j$ in $w_k$.  More precisely, if $w_k = \prod_{n=1}^N x_{a_n}^{\epsilon_n}$ with each $a_n=i$ or $j$ and each $\epsilon_n=\pm1$ then 
$$
e_{ij}(w_k) = \Sum_{1\le n<m\le N}^N \delta(a_n,i)\delta(a_m,j) \epsilon_m\epsilon_n.
$$
Here $\delta(x,y) = \begin{cases}1&x=y\\0&x\neq y\end{cases}$ is the Kroeneker delta function.  We invite the reader to use the clasp-words of \pref{eqn:claspwords} in order to verify the computation of  $e_{12}(w_3)$, $e_{23}(w_1)$, and $e_{31}(w_2)$ below:
$$\begin{array}{rcl}
e_{23}(w_1)=+1\hspace{5mm} 
e_{31}(w_2)=+1\hspace{5mm} 
e_{12}(w_3)=+1.\hspace{5mm}
\end{array}
$$
Finally,  $\mu_{123}(L) = e_{12}(w_3)+e_{23}(w_1)+e_{31}(w_2)$. Thus, for the link of Figure~\ref{fig: Sample computation} $\mu_{123}(L)=3$.  

When any linking number fails to vanish, $\mu_{123}(L)\in \Z$ fails to be well defined.  One instead writes $\overline{\mu}_{123}(L)\in \sfrac{\Z}{(\operatorname{GCD}(\lk(L_1,L_2),\lk(L_1,L_3),\lk(L_1,L_3)))}$.  Since all links considered in this paper will have vanishing linking numbers, this ambiguity will not be relevant to us. 

The key computational result of \cite{AADG20} is that these very combinatorial constructions can be interpreted geometrically.  Let $w$ be a word in $x_i$, $x_i^{-1}$, $x_j$ and $x_j^{-1}$ with $(i,j)=(1,2)$, $(2,3)$, or $(3,1)$.  The \textbf{word-curve} $\gamma_{ij}(w)$ is the piecewise linear curve in $\R^2$ defined as follows.  Start at $(0,0)$; when you see an $x_i$ in $w$ travel one to the right, when you see an $x_i^{-1}$, one to the left, $x_j$ one up and $x_j^{-1}$ one down.  Using the words of \pref{eqn:claspwords} we get word-curves as in Figure~\ref{fig: word curves}.  The astute reader will notice that the (signed) areas of the regions they enclose agree with $e_{12}(w_3)$, $e_{23}(w_1)$ and $e_{31}(w_2)$.  According to \cite{AADG20} this always happens.
\begin{theorem}[Theorem 7 of \cite{AADG20}]\label{AADG theorem}
If $w$ is a word in $x_i$ and $x_j$, then $e_{ij}(w) = \displaystyle\oint_{\gamma_{ij}(w)} x~dy$.  If $\gamma_{ij}(w)$ is a simple closed curve with the counterclockwise orientation, then this is the area of region that $\gamma_{ij}(w)$ bounds.  
\end{theorem}

\begin{figure}
\centering
     \begin{tikzpicture}
     \draw[thin] (-1.5,0)--(1.5,0); 
     \draw[thin] (0,-1.5)--(0,1.5);
     \draw (0,0)--(0,1)--(-1,1)--(-1,0)--(0,0);
     \draw[->](0,0)--(0,.5);
     \node at (0,-2) {$\gamma_{23}(x_3x_2^{-1}x_3^{-1}x_2)$};
     \end{tikzpicture}
     \hspace{.1\textwidth}
     \begin{tikzpicture}
     \draw[thin] (-1.5,0)--(1.5,0); 
     \draw[thin] (0,-1.5)--(0,1.5);
     \draw (0,0)--(0,1)--(-1,1)--(-1,0)--(0,0);
     \draw[->](0,0)--(0,.5);
     \node at (0,-2) {$\gamma_{31}(x_1x_3^{-1}x_1^{-1}x_3)$};
     \end{tikzpicture}
     \hspace{.1\textwidth}
     \begin{tikzpicture}
     \draw[thin] (-1.5,0)--(1.5,0); 
     \draw[thin] (0,-1.5)--(0,1.5);
     \draw (0,0)--(1,0)--(1,1)--(0,1)--(0,0);
     \draw[->](0,0)--(.5,0);
     \node at (0,-2) {$\gamma_{12}(x_1x_2x_1^{-1}x_2^{-1})$};
     \end{tikzpicture}
     
        \caption{The word-curves associated with the $C$-complex and basepoints of Figure~\ref{fig: Sample computation}.}
        \label{fig: word curves}
\end{figure}
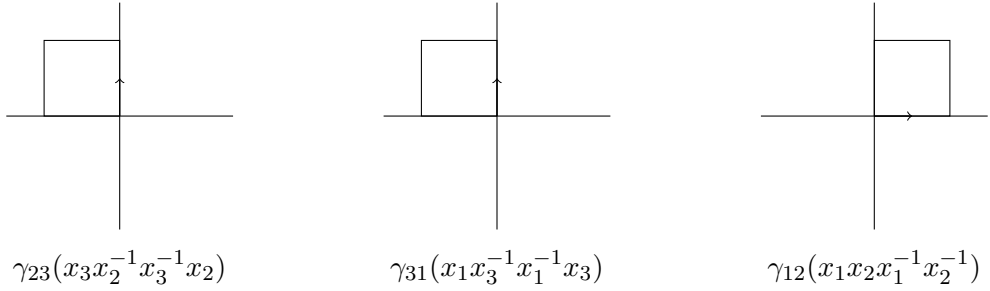

\section{A better bound on the clasp number of a link} \label{sect: lower bound}

The bound appearing in \cite[Theorem 5]{AADG20} comes from noticing that the region bounded by the word-curve $\gamma_{ij}(w_k)$ is a polyomino and that work of Harary-Harborth \cite{HH76} gives a sharp lower bound on the perimeter of a polyomino in terms of its area.  Our improvement comes from thinking about all three word curves $\gamma_{12}(w_3)$, $\gamma_{23}(w_1)$, and $\gamma_{31}(w_2)$ at once.  The words $w_1$, $w_2$ and $w_3$ are related.  For instance, the number of $x_2^{+1}$'s in $w_1$ is equal to the number of $x_1^{+1}$'s in $w_2$, as both of these count the number of positive clasps in $F_1\cap F_2$.   We encode the details of this relationship in the following definition.

\begin{definition}\label{defn: word curve triple}
A closed curve $\gamma$ in $\R^2$ is called a \textbf{word-curve} if it consists of horizontal and vertical line segments with corners only at integer lattice points.  For word curves $\gamma_1, \gamma_2, \gamma_3$ we call $(\gamma_1, \gamma_2, \gamma_3)$ a \textbf{word-curve triple} if 
\begin{center}
$\displaystyle\oint_{\gamma_{1}}|dx| = \oint_{\gamma_{2}}|dy|$, $\displaystyle\oint_{\gamma_{2}}|dx|=\oint_{\gamma_3}|dy|$, and $\displaystyle\oint_{\gamma_{3}}|dx|=\oint_{\gamma_1}|dy|$
\end{center}
\end{definition}
The connection with claspwords is that if $F$ is a C-complex, then $\displaystyle\oint_{\gamma_{12}(w_3)}|dx|$ records (without sign) how any $x_1$'s appear in $w_3$ and $\displaystyle\oint_{\gamma_{23}(w_1)}|dy|$ records how any $x_1$'s appear in $w_3$.  But these two each count the number of clasps in $F_1\cap F_3$. The following proposition encodes this fact.

\begin{proposition}
For any 3-component link $L$ with vanishing pairwise linking numbers and any C-complex $F$, $(\gamma_{12}(w_3), \gamma_{23}(w_1), \gamma_{31}(w_2))$ is a word-curve triple. 
\end{proposition}





It should not come as a surprise that if one wants to maximize the area enclosed in a curve while fixing its horizontal and vertical length, then a rectangle should be optimal.  The following makes this precise.

\begin{proposition}\label{prop: 
Relate area with x- and-y length}
Let $\gamma$ be word curve, 
then 
 $\left\lvert\displaystyle\oint_{\gamma} x\, dy\right\rvert\leq\frac{1}{4}\Oint_{\gamma}|dx|\cdot \Oint_{\gamma}|dy|$.
\end{proposition}

\begin{proof}


Since all of the relevant integrals are preserved by shifting $\gamma$ in the plane, we may arrange that the minimum $x$-value on $\gamma$ is $0$.  Next, we note that one can compute $\Oint x\, dy$ by computing it over each of its constituent line segments, 
$$
\Oint_{\gamma} x\, dy = \Sum_{\sigma\in V_+}\oint_{\sigma} x\, dy+\Sum_{\delta\in V_-}\oint_{\delta} x\, dy+\Sum_{h\in H}\oint_{h} x\, dy
$$
where $V_+$, $V_-$ and $H$ are the sets of all upward, downward and horizontally facing line segments in $\gamma$, respectively.  Over each horizontal segment, $dy=0$, so $\Sum_{h\in H}\oint_{h} x\, dy=0$.  Over each   vertical line segment, $\sigma$ or $\delta$, $x = x(\sigma)$ or $x(\delta)$ is constant and so for each upward line segment $\Oint_{\sigma} x\, dy = x(\sigma)\cdot \ell(\sigma)$, where $\ell(\sigma)$ is the length of $\sigma$.  Similarly for each downward pointing segment, $\delta$, $\Oint_{\delta} x\, dy = -x(\delta)\cdot \ell(\delta)$.  Thus
$$
    \oint_\gamma x\, dy = \sum_{\sigma\in V_+} x(\sigma)\ell(\sigma) - \sum_{\delta\in V_-} x(\delta)\ell(\delta) \leq \sum _{\sigma\in V_+}x(\sigma)\ell(\sigma) $$
where the inequality comes from noticing that $\Sum_{\delta\in V_-} x(\delta)\ell(\delta)$ is non-negative.  Each $x(\sigma)$ is less than or equal to $\max (x)$, the maximum $x$-value on $\gamma$, thus,
$$
    \oint_\gamma x\, dy \leq \max( x)\cdot\sum_{\sigma\in V_+}\ell(\sigma). $$
    As $\gamma$ is a closed curve, $\Sum_{\sigma\in V_+}\ell(\sigma) = \Sum_{\delta\in V_-}\ell(\delta)$ where the latter sum is over all downward pointing line segments.  Moreover $\Sum_{\sigma\in V_+}\ell(\sigma) + \Sum_{\delta\in V_-}\ell(\delta) = \Oint_{\gamma} |dy|$.  Thus, $\Sum_{\sigma\in V_+}\ell(\sigma) = \frac{1}{2}\Oint_{\gamma} |dy|$ and 
    $$
    \oint x\, dy \leq \frac{1}{2} \max(x) \cdot \Oint_{\gamma} |dy|. $$
    Finally, $\max (x)=\Oint_{\gamma_0} dx$ where $\gamma_0$ is an arc of $\gamma$ starting at an $x$-value of $0$ and ending at $x_{\max}$.  Next, $\Oint_{\gamma_0} dx\le \sum_{\epsilon\in H_+}\ell(\epsilon)$ where $H_+$ is the set of rightward pointing line segments in $\gamma$.  For the same reason that $\Sum_{\sigma\in V_+}\ell(\sigma) = \frac{1}{2}\Oint_{\gamma} |dy|$, we see that  $\Sum_{\epsilon\in H_+}\ell(\epsilon) = \frac{1}{2}\oint_\gamma|dx|$.  Thus,
    $$
    \oint x\, dy \leq \frac{1}{4}\Oint_{\gamma}|dx|\cdot \Oint_{\gamma}|dy|. $$
\end{proof}

Our next result is an inequality relating the signed area bounded by a word curve triple $\gamma$ (as in Theorem~\ref{AADG theorem}) and the total length of $\gamma$ For a word curve, $\gamma = (\gamma_1, \gamma_2, \gamma_3)$, we will call $M(\gamma) = \Oint_{\gamma_1}xdy+\Oint_{\gamma_2}xdy+\Oint_{\gamma_3}xdy$ the signed area bounded by $\gamma$ and $L(\gamma) = \Oint_{\gamma_1}(|dx|+|dy|)+\Oint_{\gamma_2}(|dx|+|dy|)+\Oint_{\gamma_3}(|dx|+|dy|)$ its length.   

\begin{lemma}\label{lemma: do it for rectangles}
Let $\gamma$ be a  word-curve triple.  Then $L(\gamma)\ge 4\sqrt{3|M(\gamma)|}$.
\end{lemma}

\begin{proof}
Let $\gamma=(\gamma_1, \gamma_2, \gamma_3)$ be a word curve triple, $X_i=\Oint_{\gamma_i} |dx|$ and $Y_i=\Oint_{\gamma_i} |dy|$.  Since $\gamma$ is a word curve triple, $Y_2=X_1$, $Y_3=X_2$ and $Y_1=X_3$.  Thus, by Proposition~\ref{prop: 
Relate area with x- and-y length}, 
$$
|M(\gamma)|\le\frac{1}{4}\left(X_1Y_1+X_2Y_2+X_3Y_3\right)=\frac{1}{4}\left(X_1X_3+X_2X_1+X_3X_2\right).
$$
 Appealing to the Cauchy-Schwartz inequality, 
$$
X_1X_3+X_2X_1+X_3X_2 \le \sqrt{(X_1^2+X_2^2+X_3^2)(X_3^2+X_1^2+X_2^2)} = X_1^2+X_2^2+X_3^2.
$$
Thus,
\begin{eqnarray*}X_1X_3+X_2X_1+X_3X_2 &= &\frac{1}{3}(X_1X_3+X_2X_1+X_3X_2)+\frac{2}{3}(X_1X_3+X_2X_1+X_3X_2)
\\&\le&\frac{1}{3}(X_1^2+X_2^2+X_3^2)+\frac{2}{3}(X_1X_3+X_2X_1+X_3X_2) 
\\&=& \frac{1}{3}(X_1+X_2+X_3)^2.
\end{eqnarray*}
Again using that $\gamma$ is a word curve triple, $X_1+X_2+X_3=\frac{1}{2}L(\gamma)$, and so
$$
|M(\gamma)|
\le \frac{1}{12}(X_1+X_2+X_3)^2
\le\frac{1}{12} \left(\frac{1}{2}L(\gamma)\right)^2.
$$
Solving for $L(\gamma)$, 
$$
4\sqrt{3|M(\gamma)|}\le L(\gamma).
$$ \end{proof}

We are now ready to prove Theorem~\ref{thm: calc}.

\begin{reptheorem}{thm: calc}
Let \(L=L_1\cup L_2\cup L_3\) be a 3-component link with vanishing pairwise linking numbers. Then, \(C(L)\geq2\lceil{\sqrt{3|\mu_{123}(L)|}}\rceil\). 
\end{reptheorem}

\begin{proof}
Let $L$ be a 3-component link. Let $F$ be a C-complex for $L$ and $\gamma=(\gamma_{12}(w_3), \gamma_{23}(w_1), \gamma_{31}(w_2))$ be the resulting word curve triple.  As each clasp in $F_i\cap F_j$ contributes a letter to the clasp words $w_i$ and $w_j$, and each of these contribute a length 1 line segment to $\gamma(w_i)$ and $\gamma(w_j)$, $L(\gamma)$ is two times $C(F)$, the number of clasps in $F$. By Theorem~\ref{AADG theorem} $M(\gamma) = \mu_{123}(L)$.

By Lemma~\ref{lemma: do it for rectangles}, then 
$$
2C(F)=L(\gamma)\ge 4\sqrt{3|\mu_{123}(L)|}.
$$
Division by 2 and taking the minimum of $C(F)$ amongst all C-complexes for $L$ completes the proof. 
\end{proof}

\section{examples of links with small clasp number}\label{sect:examples}

In this section we explain how to produce $3$-component links $L$ realizing the bound of Theorem~\ref{thm: calc}.  In particular, by the end of this analysis we will have produced the links of Figure~\ref{fig: Optimal example}.     We begin with a lemma which decribes the casewise structure of this bound.  

\begin{lemma}\label{lem:casewise}
   Let $M,k\in\mathbb{N}\cup\{0\}$ and $3k^2<M\leq3(k+1)^2$.  Then  $$2\ceil{\sqrt{3M}}=\begin{cases}
    6k+2 & 3k^2<M\leq 3k^2+2k \\
    6k+4 & 3k^2+2k<M\leq 3k^2+4k+1 \\
    6(k+1) & 3k^2+4k+1<M\leq 3(k+1)^2.
\end{cases}$$
\end{lemma}
\begin{proof}
We start by giving the two sides of the claimed equality names, let $f(M)=2\ceil{\sqrt{3M}}$ and 
\\
$g(M)=\begin{cases}
    6k+2 & 3k^2<M\leq 3k^2+2k \\
    6k+4 & 3k^2+2k<M\leq 3k^2+4k+1 \\
    6(k+1) & 3k^2+4k+1<M\leq 3(k+1)^2
\end{cases}$.  Our argument is motivated by looking at the jumps of these functions. Let $J_f(M)=f(M+1)-f(M)$ and $J_g(M)=g(M+1)-g(M)$.  By its very definition $J_g(M)=0$ unless $M=3k^2$ or $3k^2+2k$ or $3k^2+4k+1$ with $k\in \N$.  By inspecting the graph of Figure~\ref{fig:graph}, we see that the jumps of $f$ seem to agree with the jumps of $g$.

\begin{figure}[htb]
\begin{tikzpicture}[scale=.3]
\draw(0,0)--(23,0);
\draw(0,0)--(0,18);
   \foreach \M in {1,...,23}
   {\draw[fill=black] ({\M},{2*ceil(sqrt(3*\M))}) circle(.1);
   \draw({\M},0)--({\M},.5);}
   \foreach \M in {1,...,17}
   {\draw(0,{\M})--(.5,{\M});}
   \foreach \K in {1,2}
   {
   \draw[dashed] (3*\K*\K,0)--(3*\K*\K,18);
   \node[rotate=90,right] at (3*\K*\K+.5,1) {\tiny{$x=3\cdot\K^2$}};
   \draw[dashed] (3*\K*\K+2*\K,0)--(3*\K*\K+2*\K,18);
   \node[rotate=90,right] at (3*\K*\K+2*\K+.5,1) {\tiny{$x=3\cdot\K^2+2\cdot\K$}};
   \draw[dashed] (3*\K*\K+4*\K+1,0)--(3*\K*\K+4*\K+1,18);
   \node[rotate=90,right] at (3*\K*\K+4*\K+1+.5,1) {\tiny{$x=3\cdot\K^2+4\cdot\K+1$}};
   }
   \draw[dashed] (1,0)--(1,18);
   \node[rotate=90] at (1+.5,12) {\tiny{$x=3\cdot0^2+4\cdot0+1$}};
   
\end{tikzpicture}
\caption{$y=2\ceil{\sqrt{3M}}$, with vertical lines where $J_g(x)>0$.}\label{fig:graph}
\end{figure}
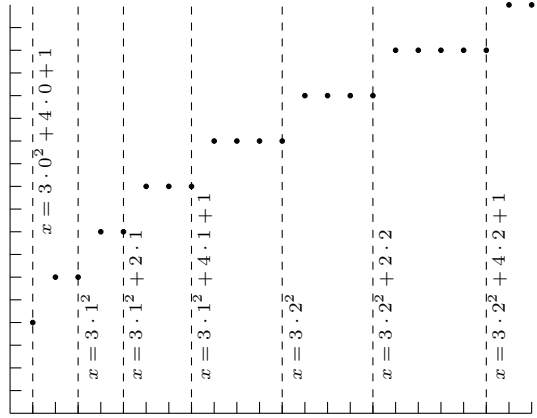

Verifying this, $J_f(M)\neq 0$ if and only if there is a perfect square in the interval $[3M,3(M+1))$.  Thus we have three cases, either $3M\le (3k)^2<3(M+1)$, $3M\le (3k+1)^2<3(M+1)$, or $3M\le (3k+2)^2<3(M+1)$, 
In each case $M=\floor{\sfrac{9k^2}{3}} = 3k^2$, 
$M=\floor{\sfrac{(3k+1)^2}{3}}=3k^2+2k$, 
or $M=\floor{\sfrac{(3k+2)^2}{3}} =3k^2+4k+1$.  This agrees with $J_g(M)$.  An induction with base case $M=3k^2$ completes the proof.\end{proof}

\begin{reptheorem}{thm:realization}
For any $M\in \Z$,  there exists a link L such that $\mu_{123}(L)=M$ and \(C(L)=2\ceil{\sqrt{3|M|}}.\) 
\end{reptheorem}
\begin{proof}
Since $\mu_{123}(L)$ changes by a sign if we reverse the orientation of any one component, it suffices to prove the theorem when $M\ge 0$.  When $M=0$ the needed example is given by the unlink, so we may assume $M\ge 1$.  The reader seeking efficiency is now invited to directly count the number of clasps and compute the triple linking numbers of the links of Figure~\ref{fig: Optimal example}.  After doing so and consulting with Lemma~\ref{lem:casewise}, the proof will be complete.  We will take a longer process to motivate these links.  

  We need to produce a set of instructions given a natural number $M$ yielding a link with $\mu_{123}(L)=M$ and \(C(L)=2\ceil{\sqrt{3M}}\).  We will find such a link by first producing a word curve triple with $M(\gamma)=M$ and $L(\gamma) = 4\ceil{\sqrt{3M}}$.   
 
Our proof is motivated by the cases of Lemma~\ref{lem:casewise}.     In the case $3k^2<M\le 3k^2+2k$, there are some $0\le \ell_1\le k$ and $0<\ell_2\le k$ so that $M=3k^2+\ell_1+\ell_2$.  The needed word curve now appears in Figure~\ref{subfig: Case 1}.  To construct this curve, start with three $k\times k$ squares and then add a $1\times \ell_1$ rectangle to the first and an $\ell_2\times 1$ rectangle to the second, so the the boundaries of these three curves still satisfy Definition~\ref{defn: word curve triple}.  By counting the area bounded by $\gamma$, $M(\gamma)  =k^2+\ell_1+k^2+\ell_2+k^2 = M$, and $L(\gamma) = 12k+4$, glancing at Lemma~\ref{lem:casewise} confirms that $L(\gamma) = 4\ceil{\sqrt{3M}}$, as desired.

The cases $3k^2+2k<M\le 3k^2+4k+1$ and $3k^2+4k+1<M\le 3(k+1)^2$, similarly appear in Figures~\ref{subfig: Case 2} and \ref{subfig: Case 3} respectively.  The same analysis as above reveals that $M(\gamma)=M$ and $L(\gamma) = 4\ceil{\sqrt{3M}}$.  In our next step, we translate these word curves into claspwords.

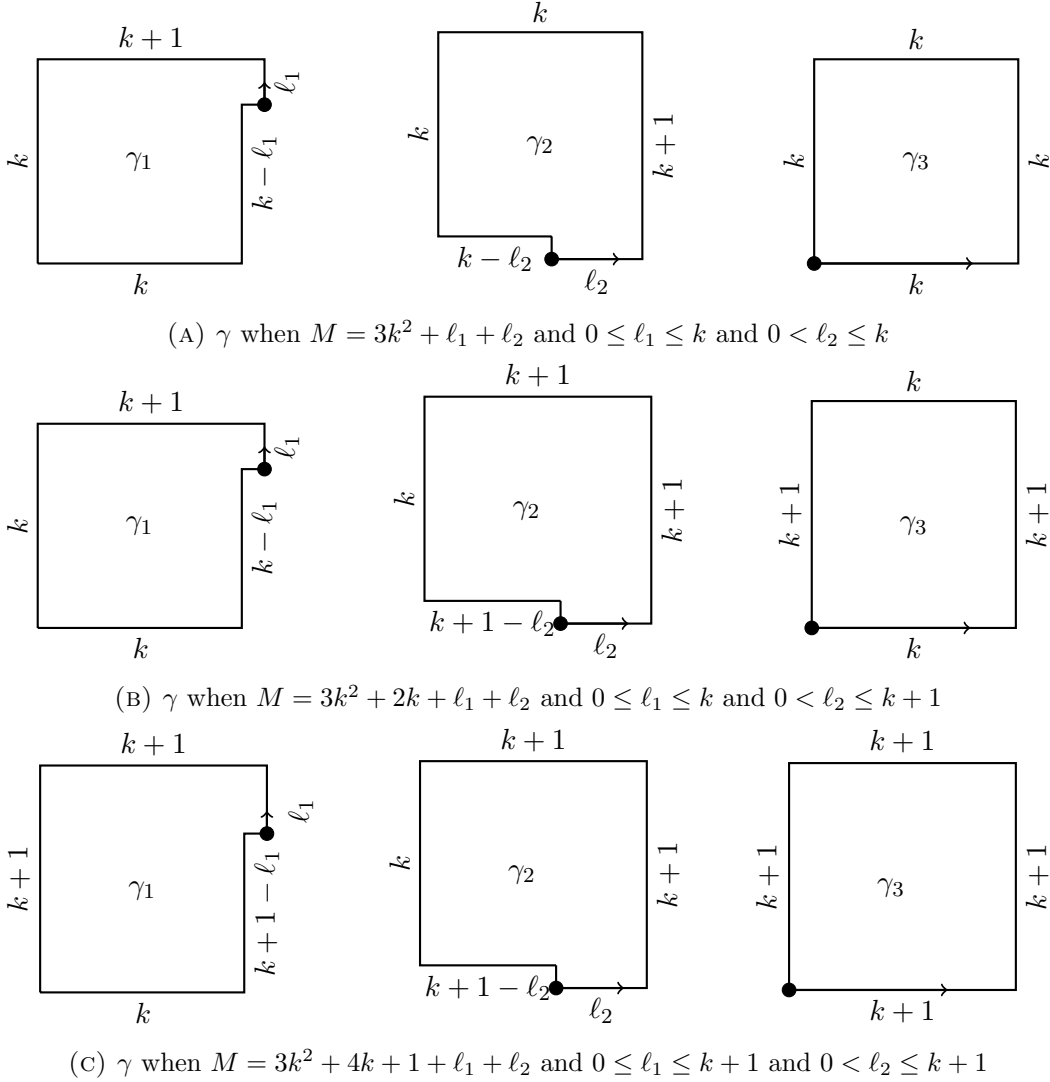
\begin{figure}[htb]

\begin{subfigure}[b]{0.85\textwidth}
\begin{tikzpicture}[scale=.3]
\draw[fill=black] (10,7) circle(.3);
\draw[->][thick] (10,7)--(10,8);
\draw[thick](0,0)--(9,0)--(9,7)--(10,7)--(10,9)--(0,9)--(0,0);
\node[below] at (4.55,0) {$k$};
\node[below, rotate=90] at (9,4) {$k-\ell_1$};
\node[below, rotate=90] at (10,8) {$\ell_1$};
\node[above] at (5,9) {$k+1$};
\node[above, rotate=90] at (0,4.5) {$k$};
\node at (4.5,4.5) {$\gamma_1$};
\end{tikzpicture}
\hfill
\begin{tikzpicture}[scale=.3]
\draw[thick](0,0)--(0,-1)--(4,-1)--(4,9)--(-5,9)--(-5,0)--(0,0);
\draw[fill=black] (0,-1) circle(.3);
\draw[->][thick] (0,-1)--(3,-1);
\node[below] at (2,-1) {$\ell_2$};
\node[below, rotate=90] at (4,4) {$k+1$};
\node[above] at (-.5,9) {$k$};
\node[above, rotate=90] at (-5,4.5) {$k$};
\node[below] at (-2.5,0) {$k-\ell_2$};
\node at (-.5,4) {$\gamma_2$};
\end{tikzpicture}
\hfill
\begin{tikzpicture}[scale=.3]
\draw[thick](0,0)--(9,0)--(9,9)--(0,9)--(0,0);
\node[below] at (4.5,0) {$k$};
\node[below, rotate=90] at (9,4.5) {$k$};
\node[above] at (4.5,9) {$k$};
\node[above, rotate=90] at (0,4.5) {$k$};
\node at (4.5,4.5) {$\gamma_3$};
\draw[fill=black] (0,0) circle(.3);
\draw[thick, ->] (0,0)--(7,0);
\end{tikzpicture}
\subcaption{$\gamma$ when $M=3k^2+\ell_1+\ell_2$  and $0\le \ell_1\le k$ and $0<\ell_2\le k$}\label{subfig: Case 1}
\end{subfigure}

\begin{subfigure}[b]{0.85\textwidth}
\begin{tikzpicture}[scale=.3]
\draw[fill=black] (10,7) circle(.3);
\draw[->][thick] (10,7)--(10,8);
\draw[thick](0,0)--(9,0)--(9,7)--(10,7)--(10,9)--(0,9)--(0,0);
\node[below] at (4.55,0) {$k$};
\node[below, rotate=90] at (9,4) {$k-\ell_1$};
\node[below, rotate=90] at (10,8) {$\ell_1$};
\node[above] at (5,9) {$k+1$};
\node[above, rotate=90] at (0,4.5) {$k$};
\node at (4.5,4.5) {$\gamma_1$};
\end{tikzpicture}
\hfill
\begin{tikzpicture}[scale=.3]
\draw[thick](1,0)--(1,-1)--(5,-1)--(5,9)--(-5,9)--(-5,0)--(1,0);
\draw[fill=black] (1,-1) circle(.3);
\draw[->][thick] (1,-1)--(4,-1);
\node[below] at (3,-1) {$\ell_2$};
\node[below, rotate=90] at (5,4) {$k+1$};
\node[above] at (0,9) {$k+1$};
\node[above, rotate=90] at (-5,4.5) {$k$};
\node[below] at (-2,0) {$k+1-\ell_2$};
\node at (-.5,4) {$\gamma_2$};
\end{tikzpicture}
\hfill
\begin{tikzpicture}[scale=.3]
\draw[thick](0,0)--(9,0)--(9,10)--(0,10)--(0,0);
\node[below] at (4.5,0) {$k$};
\node[below, rotate=90] at (9,5) {$k+1$};
\node[above] at (4.5,10) {$k$};
\node[above, rotate=90] at (0,5) {$k+1$};
\node at (4.5,4.5) {$\gamma_3$};
\draw[fill=black] (0,0) circle(.3);
\draw[thick, ->] (0,0)--(7,0);
\end{tikzpicture}
\subcaption{$\gamma$ when $M=3k^2+2k+\ell_1+\ell_2$  and $0\le \ell_1\le k$ and $0<\ell_2\le k+1$}\label{subfig: Case 2}
\end{subfigure}

\begin{subfigure}[b]{0.85\textwidth}
\begin{tikzpicture}[scale=.3]
\draw[fill=black] (10,7) circle(.3);
\draw[->][thick] (10,7)--(10,8);
\draw[thick](0,0)--(9,0)--(9,7)--(10,7)--(10,10)--(0,10)--(0,0);
\node[below] at (4.55,0) {$k$};
\node[below, rotate=90] at (9,3.5) {$k+1-\ell_1$};
\node[below, rotate=90] at (10.5,8) {$\ell_1$};
\node[above] at (5,10) {$k+1$};
\node[above, rotate=90] at (0,5) {$k+1$};
\node at (4.5,4.5) {$\gamma_1$};
\end{tikzpicture}
\hfill
\begin{tikzpicture}[scale=.3]
\draw[thick](1,0)--(1,-1)--(5,-1)--(5,9)--(-5,9)--(-5,0)--(1,0);
\draw[fill=black] (1,-1) circle(.3);
\draw[->][thick] (1,-1)--(4,-1);
\node[below] at (3,-1) {$\ell_2$};
\node[below, rotate=90] at (5,4) {$k+1$};
\node[above] at (0,9) {$k+1$};
\node[above, rotate=90] at (-5,4.5) {$k$};
\node[below] at (-2,0) {$k+1-\ell_2$};
\node at (-.5,4) {$\gamma_2$};
\end{tikzpicture}
\hfill
\begin{tikzpicture}[scale=.3]
\draw[thick](0,0)--(10,0)--(10,10)--(0,10)--(0,0);
\node[below] at (5,0) {$k+1$};
\node[below, rotate=90] at (10,5) {$k+1$};
\node[above] at (5,10) {$k+1$};
\node[above, rotate=90] at (0,5) {$k+1$};
\node at (4.5,4.5) {$\gamma_3$};
\draw[fill=black] (0,0) circle(.3);
\draw[thick, ->] (0,0)--(7,0);
\end{tikzpicture}
\subcaption{$\gamma$ when $M=3k^2+4k+1+\ell_1+\ell_2$ and $0\le \ell_1\le k+1$ and $0<\ell_2\le k+1$}\label{subfig: Case 3}
\end{subfigure}

\caption{The word curve triples, $\gamma=(\gamma_1,\gamma_2,\gamma_3)$ enclosing total area $M$ while minimizing total length.  All of these are oriented counterclockwise.}\label{fig: needed word curve triples}
\end{figure}

  In the case of Figure~\ref{subfig: Case 1}, we start at the decorated point and read off a claspword
  \begin{equation*}\label{eqn:case1}
  \begin{array}{l}
  w_{1}=x_3^{\ell_1}x_2^{-k-1}x_3^{-k}x_2^kx_3^{k-\ell_1}x_2,
  \\
  w_{2}=x_3^{\ell_2}x_1^{k+1}x_3^{-k}x_1^{-k}x_3^{k-\ell_2}x_1^{-1},~\text{and}
  \\
  w_{3}=x_1^kx_2^kx_1^{-k}x_2^{-k}.
  \end{array}
  \end{equation*}
  Similarly, in the case of Figure~\ref{subfig: Case 2}, 
   \begin{equation*}\label{eqn:case2}
  \begin{array}{l}
  w_{1}=x_3^{\ell_1}x_2^{-k-1}x_3^{-k}x_2^kx_3^{k-\ell_1}x_2,
  \\
  w_{2}=x_3^{\ell_2}x_1^{k+1}x_3^{-k-1}x_1^{-k}x_3^{k+1-\ell_2}x_1^{-1},~\text{and}
  \\
  w_{3}=x_1^kx_2^{k+1}x_1^{-k}x_2^{-k-1}.
  \end{array}
  \end{equation*}
  In the case of Figure~\ref{subfig: Case 3}, 
   \begin{equation*}\label{eqn:case3}
  \begin{array}{l}
  w_{1}=x_3^{\ell_1}x_2^{-k-1}x_3^{-k-1}x_2^kx_3^{k+1-\ell_1}x_2,
  \\
  w_{2}=x_3^{\ell_2}x_1^{k+1}x_3^{-k-1}x_1^{-k}x_3^{k+1-\ell_2}x_1^{-1},~\text{and}
  \\
  w_{3}=x_1^{k+1}x_2^{k+1}x_1^{-k-1}x_2^{-k-1}.
  \end{array}
  \end{equation*}

 Notice these these can be combined into a single case as 

  \begin{equation}\label{eqn:combineCase}
  \begin{array}{l}
  w_{1}=x_3^{\ell_1}x_2^{-\alpha-1}x_3^{-\beta}x_2^\alpha x_3^{\beta-\ell_1}x_2,
  \\
  w_{2}=x_3^{\ell_2}x_1^{\alpha+1}x_3^{-\gamma}x_1^{-\alpha}x_3^{\gamma-\ell_2}x_1^{-1},~\text{and}
  \\
  w_{3}=x_1^\beta x_2^\gamma x_1^{-\beta}x_2^{-\gamma}.
  \end{array}
  \end{equation}
  where $(\alpha,\beta,\gamma)\in\{(k,k,k), (k,k,k+1),(k,k+1, k+1)\}$ and $0\le \ell_1\le \beta$ and $1\le \ell_2\le \gamma$. It is now straightforward to construct three disks intersecting in clasps to result in these claspwords.  One such example appears in Figure~\ref{fig: First Case}, completing the proof.  \end{proof}

\bibliographystyle{alpha}
\bibliography{biblio}  

\end{document}